\documentclass{amsart}
\usepackage{tikz}
\usetikzlibrary{decorations.pathreplacing}
\usetikzlibrary{decorations.pathmorphing}
\pgfdeclarelayer{bg}
\pgfdeclarelayer{edges}
\pgfsetlayers{bg,edges,main}
\newenvironment{edges}{\begin{pgfonlayer}{edges}}{\end{pgfonlayer}}
\tikzstyle{vertex}=[inner sep = 0pt, minimum width=4pt, draw=black, fill=black, shape=circle]
\renewcommand{\subset}{\subseteq}
\newcommand{\gpoint}[2]{\node[style=vertex, label=#1:$#2$]}
\newcommand{\apoint}[1]{\gpoint{above}{#1}}
\newcommand{\bpoint}[1]{\gpoint{below}{#1}}
\newcommand{\rpoint}[1]{\gpoint{right}{#1}}
\newcommand{\makecvert}[1]{\draw[red, shift only] (#1) circle (5pt)}
\tikzstyle{blueedge}=[draw=blue, decorate, decoration={snake, segment length=.2cm, amplitude=.05cm}]
\tikzstyle{rededge}=[draw=red]
\tikzstyle{medge}=[draw=red, line width=2pt]
\tikzstyle{reddots}=[dashed, draw=red]
\usepackage{enumerate}
\usepackage{xcolor}
\usepackage{graphicx}
\usepackage[normalem]{ulem}
\newcommand{\gp}[1]{\textcolor{black}{#1}}
\newcommand{\totdom}{\bar{\gamma}}

\newcommand{\cee}{\mathcal{C}}

\newtheorem{proposition}{Proposition}[section]
\newtheorem{lemma}[proposition]{Lemma}
\newtheorem{claim}{Claim}
\newtheorem{theorem}[proposition]{Theorem}
\newtheorem{corollary}[proposition]{Corollary}
\newtheorem{conjecture}[proposition]{Conjecture}
\newtheorem{ques}[proposition]{Question}
\theoremstyle{definition}

\theoremstyle{remark}

\newcommand{\st}{\colon\,}
\newcommand{\sizeof}[1]{\left\lvert{#1}\right\rvert}
\newcommand{\floor}[1]{\left\lfloor{#1}\right\rfloor}

\newcommand{\xee}{\mathcal{X}}
\newcommand{\tee}{\mathcal{T}}

\newcommand{\yee}{\mathcal{Y}}
\newcommand{\nt}{n_t}
\newcommand{\nl}{n_{\ell}}
\newcommand{\nf}{n_4}
\newcommand{\nx}{n_0}

\newenvironment{proofclaim}{\begin{proof}[Proof of Claim]}{\end{proof}}
\title[Strong coloring $2$-regular graphs]{Strong coloring $2$-regular graphs: cycle restrictions and partial colorings}
\author{Jessica McDonald}
\address{Department of Mathematics and Statistics, Auburn University, Auburn, AL 36849}
\email{mcdonald@auburn.edu}
\thanks{Supported in part by NSF grant DMS-1600551.}
\author{Gregory J. Puleo}
\address{Department of Mathematics and Statistics, Auburn University, Auburn, AL 36849}
\email{gjp0007@auburn.edu}
\begin{document}
\maketitle

\begin{abstract} Let $H$ be a graph with $\Delta(H) \leq 2$, and let
  $G$ be obtained from $H$ by gluing in vertex-disjoint copies of
  $K_4$. We prove that if $H$ contains at most one odd cycle of length
  exceeding $3$, or if $H$ contains at most $3$ triangles, then
  $\chi(G) \leq 4$. This proves the \emph{Strong Coloring Conjecture}
  for such graphs $H$. For graphs $H$ with $\Delta=2$ that are not
  covered by our theorem, we prove an approximation result towards the
  conjecture.
\end{abstract}

\section{Introduction}

In this paper all graphs are assumed to be simple, unless explicitly
stated otherwise. The reader is referred to \cite{WestText} for
standard terminology. One non-standard term we use is the idea of
\emph{gluing} one graph onto another. Given vertex-disjoint graphs
$G_1, \ldots, G_q$, $H$ with
$|\bigcup _{1\leq i\leq q} V(G_1)|\leq |V(H)|$, we \emph{glue
  $G_1, \ldots, G_q$ onto $H$} by defining an injective function
$f:\cup_{1\leq i\leq q} V(G_i)\rightarrow V(H)$, and then forming a
new graph $G$ with $V(G) = V(H)$ and $E(G) = E(H) \cup \bigcup_{1\leq i \leq q} E_i,$
where
$E_i  =\{f(a)f(b)\st ab\in E(G_i), f(a)f(b)\not\in E(H)\}.$
The graph $G$ is said to have been obtained from $H$ by \emph{gluing
  in} $G_1, \ldots, G_q$.

In this paper we are primarily concerned with the chromatic number of a graph obtained by gluing vertex-disjoint copies of $K_4$ onto a 2-regular graph. To contextualize this, consider the following general question.

\begin{ques}\label{ques:Kt} Suppose that $G$ is obtained from a 2-regular graph $H$ by gluing in some number of vertex-disjoint copies of $K_t$. Is $\chi(G)\leq t$?
\end{ques}

When $H$ is a cycle and $t=3$,  Question \ref{ques:Kt} is the famous ``cycle plus triangles problem'' originating in the work of Du, Hsu, and Hwang \cite{DHH}, popularized by Erd\H{o}s, and resolved affirmatively by Fleischner and
Stiebitz~\cite{cycletri} (see \cite{cycletri} for more on the history of this particular problem). The result for $t=3$ does not hold for all 2-regular $H$; in particular a $C_4$ component in $H$ can allow $K_4$ to be created after gluing in triangles, and Fleischner and Stiebitz \cite{fleischner-remarks} found an infinite family of counterexamples without such $C_4$ components as well, answering a further question of Erd\H{o}s ~\cite{erdos-fav}.

Haxell \cite{haxell-strong} has answered Question \ref{ques:Kt} affirmatively whenever $t\geq 5$ (in fact she proved something stronger, as we shall discuss shortly). Question \ref{ques:Kt} remains open for $t=4$ however. There is an easy affirmative argument for $t=4$ when $H$ consists only of cycles of length 3 or 4, and the problem can also be resolved positively when $H$ has girth at least 4 (see Pei \cite{Pei}). In the present paper we step into the intermediate ground, where $H$ has both triangles and longer odd cycles, and prove the following result.

\begin{theorem}\label{thm:main-cycles}
  Let $H$ be a graph with $\Delta(H) \leq 2$, and let $G$ be obtained
  from $H$ by gluing in vertex-disjoint copies of $K_4$. If $H$
  contains at most one odd cycle of length exceeding $3$, or if $H$
  contains at most $3$ triangles, then $\chi(G) \leq 4$.
\end{theorem}

We also prove the following approximation result for graphs $H$ not dealt with by the above theorem.

\begin{theorem}\label{thm:main-fraction}
  Let $H$ be a graph with $\Delta(H) \leq 2$, and let $G$ be obtained
  from $H$ by gluing in vertex-disjoint copies of $K_4$. Then there
  is a set of vertices $Z$ with $\sizeof{Z} \leq \sizeof{V(G)}/22$
  such that $\chi(G-Z) \leq 4$.
\end{theorem}


Theorem \ref{thm:main-cycles} actually proves the \emph{Strong
  Coloring Conjecture} for graphs $H$ of the sort we describe in the
theorem. Stating this conjecture in full generality requires some
technical setup which we shall now begin to give.






Given an $n$-vertex graph $H$ where $n$ is divisible by
  $t$, we say that $H$ is \emph{strongly $t$-colorable} if, for any
partition of $V(H)$ into parts of size $t$, $H$ has a $t$-coloring
where each color class is a transversal of the partition (i.e. where
each color class contains exactly one vertex from each part of the
partition). In the case where $n$ is not divisible by $t$, we say
that $H$ is \emph{strongly t-colorable} if $H'$ is strongly
  $t$-colorable, where $H'$ is obtained from $H$ by adding
$t\lceil\tfrac{n}{t}\rceil-n$ isolated vertices (the minimum amount to
ensure divisibility by $t$). The notion of strong coloring was
introduced independently by Alon \cite{Al} and Fellows \cite{Fel}
about thirty years ago.

In the definition of strong coloring, instead of requesting that each
color class is a transversal of the partition, we can equivalently
ask for a copy of $K_t$ to be glued to each part of the
partition, and then ask for the resulting graph to be
$t$-colorable. Hence, given an $n$-vertex graph $H$ with $t\mathrel{|}n$, Question
\ref{ques:Kt} is exactly asking whether or not $H$ is strongly
$t$-colorable. However, if $t\mathrel{\not|} n$, an affirmative answer to
Question \ref{ques:Kt} may not imply strong
$t$-colorability. 
In particular, Fleischner and Stiebitz's \cite{cycletri} result implies
that cycles with length divisible by 3 are strongly 3-colorable, but
it is not true that all cycles have this property (since by adding
$K_3$'s to $C_4$ plus two necessary isolates we can create $K_4$). On
the other hand, since we only require $\Delta(H)\leq 2$ in Theorem
\ref{thm:main-cycles}, we get the following as an immediate corollary.

\begin{corollary}\label{cor:strong}
  Let $H$ be a graph with $\Delta(H) \leq 2$ which either
  contains at most one odd cycle of length exceeding $3$, or
  contains at most $3$ triangles. Then $H$ is strongly 4-colorable.
\end{corollary}

It is not obvious that a strongly $t$-colorable graph is necessarily
strongly $(t+1)$-colorable, but in fact this can be shown using a
short argument due to Fellows \cite{Fel}. Given this, it makes sense
to define the \emph{strong chromatic number} of $H$, $s\chi(H)$,
as the minimum $t$ such that $H$ is strongly $t$-colorable. Note that
for any graph $H$, $s\chi(H)\geq \Delta+1,$ since if a clique was
added to the neighborhood of a $\Delta$-vertex, a copy of
$K_{\Delta+1}$ would be created in the new graph, and obviously that
is not $\Delta$-colorable. The previously-alluded to result by Haxell
\cite{haxell-strong} says that for any graph $H$,
$s\chi(H)\leq 3\Delta-1$. When $\Delta=2$ this says that
$s\chi(H)\leq 5$ (hence answering Question \ref{ques:Kt} affirmatively
for $t\geq5$). Fleischner and Stiebitz \cite{fleischner-remarks} have
given, for each $\Delta$, an example of a $\Delta$-regular graph $H$
for which $s\chi(H) \geq 2\Delta$, and hence the following conjecture
would be best possible if it is true. (Attribution for this conjecture
is somewhat tricky -- according to \cite{ABS}, it may have first
appeared explicitly in a 2007 paper by Aharoni, Berger and Ziv
\cite{ABZ} after being ``folklore'' for a while, although it could
also be considered implicit in the 2004 paper of
  Haxell~\cite{haxell-strong}.)

\begin{conjecture}[Strong Coloring Conjecture]\label{conj:scc} For any graph $H$, $s\chi(H)\leq 2\Delta(H)$.
\end{conjecture}

The Strong Coloring Conjecture is trivial for $\Delta=1$, where it
asks essentially for the union of two matchings to be bipartite.  The
conjecture remains open for all $\Delta \geq 2$, but several partial
results are known for general $\Delta$.  Haxell~\cite{haxell-improved}
proved that $s\chi(H) \leq (\tfrac{11}{4}+\varepsilon)\Delta(H)$ when
$\Delta(H)$ is sufficiently large, improving her general upper bound
of $s\chi(H) \leq 3\Delta(H)-1$. Aharoni, Berger and Ziv~\cite{ABZ}
proved a fractional-coloring version of the conjecture.  Lo and
Sanhueza-Matamala~\cite{LS} proved the following asymptotic version
of the conjecture: for any constants $c, \epsilon > 0$, there
  exists an integer $n_0$ such that $s\chi(H) \leq (2+\epsilon)\Delta(H)$ for every graph $H$ with
  $\sizeof{V(H)} \geq n_0$ and $\Delta(H) \geq cn$.

In contrast to many of the above results, which weaken the conclusion
of Conjecture~\ref{conj:scc} in some form or another, our Corollary
\ref{cor:strong} proves the exact conclusion of
Conjecture~\ref{conj:scc} in several new cases.
Corollary~\ref{cor:strong} improves previous work by Fleischner and
Stiebitz \cite{fleischner-remarks} who (separately from their cycle +
triangles solution) verified the conjecture for all cycles $H$. Our
result also strengthens the previously discussed results of Pei
\cite{Pei}, who verified Conjecture \ref{conj:scc} when $H$ consists
only of cycles of length 3 or 4, or has girth at least 4.

If $H$ is a graph and $V_1, \ldots, V_n$ are
disjoint subsets of $V(H)$, an \emph{independent set of representatives (ISR)}
of $(V_1, \ldots, V_n)$ is an independent set $R$ containing exactly
one vertex from each set $V_i$.  Hence in the definition of strongly $t$-colorable, we may replace ``$H$ has a $t$-coloring where each color class is a transversal of the partition'' with ``$H$ has $t$ disjoint ISRs of the partition''.  The following theorem of Haxell \cite{haxell-maxdeg} (proved for a general $\Delta$ but stated here for the case $\Delta=2$) guarantees the existence of one ISR (where Conjecture \ref{conj:scc} asks for four).
\begin{theorem}[Haxell~\cite{haxell-maxdeg}]\label{thm:haxell-maxdeg}
  If $H$ is a graph with $\Delta(H)\leq 2$ and $V_1, \ldots, V_n$
  are disjoint subsets of $V(H)$ with each $\sizeof{V_i} \geq 4$,
  then $(V_1, \ldots, V_n)$ has an ISR.
\end{theorem}
We shall use Theorem \ref{thm:haxell-maxdeg} in our proof of Theorem \ref{thm:main-fraction}. It is worth noting that the following extension of Theorem \ref{thm:haxell-maxdeg} serves as another approximation towards Conjecture \ref{cor:strong} when $\Delta=2$.

\begin{theorem}\label{thm:main-2isr}
  If $H$ is a graph with $\Delta(H)\leq 2$ and $V_1, \ldots, V_n$
  are disjoint subsets of $V(H)$ with each $\sizeof{V_i} = 4$,
  then $(V_1, \ldots, V_n)$ has two disjoint ISRs.
\end{theorem}

When $\Delta(H)\geq 3$, Theorem \ref{thm:main-2isr} can de deduced from the work of Aharoni, Berger, and Spr\"{u}ssel \cite{ABS} as a consequence of a more general result about matroids. Haxell observed (personal communication)
that the $\Delta(H) = 2$ case follows quickly from a strengthened
version of Theorem~\ref{thm:haxell-maxdeg}, which uses topological tools. It is also possible to give an elementary algorithmic proof of Theorem \ref{thm:main-2isr}, using a lemma of  Haxell, Szab\'o, and Tardos (Lemma~2.6 of~\cite{HST}), and we include this as an appendix to the present paper.

The rest of the paper is organized as follows. In
Section~\ref{sec:long-odd} we show how to $4$-color a large fraction
of the vertices when $H$ has few long odd cycles. In
Section~\ref{sec:isr-lems} we prove two lemmas about ISRs that are
needed for Section~\ref{sec:few-tris}, in which we show how to
$4$-color a large fraction of the vertices when $H$ has few
triangles. In Section~\ref{sec:final} we combine results from
Sections~\ref{sec:long-odd} and ~\ref{sec:few-tris} to establish both Theorem~\ref{thm:main-cycles} and Theorem~\ref{thm:main-fraction}.

\section{Graphs with few long odd cycles}\label{sec:long-odd}

Our goal in this section is to prove the following theorem.
\begin{theorem}\label{thm:longodd-910}
  Let $H$ be a graph with $\Delta(H) \leq 2$, and let $G$
  be a graph obtained from $H$ by gluing in vertex-disjoint
  copies of $K_4$. Let $\cee$ be the set of odd cycles in $H$
  with length exceeding $3$, and let $V(\cee)$ be the set
    of vertices contained in these cycles. Then there is a set of vertices $Z \subset V(\cee)$
  with $\sizeof{Z} \leq \sizeof{\cee}/2$ such that $G-Z$
  is $4$-colorable and $Z$ contains at most one vertex
  from each cycle of $\cee$. It also holds that $\sizeof{V(G) - Z} \geq (9/10)\sizeof{V(G)}$.
\end{theorem}

We will need the following lemma about equitable coloring in
our proof of Theorem \ref{thm:longodd-910}.
\begin{lemma}[Hajnal--Szemer\'edi~\cite{equitable}]\label{lem:equitable}
  If $G$ is a graph and $k > \Delta(G)$, then $G$ has a proper
  $k$-coloring with color classes $A_1, \ldots, A_k$ such that
  $\sizeof{A_i - A_j} \leq 1$ for all $i,j$.
\end{lemma}

\begin{proof}[Proof of Theorem~\ref{thm:longodd-910}]
  We start by showing that it suffices to prove the result for
  problem instances that satisfy the following additional assumptions:
  \begin{enumerate}[(1)]
  \item The vertices of each added copy of $K_4$ form an independent set in $H$,
  \item $H$ is $2$-regular, and
  \item The added copies of $K_4$ partition $V(H)$.
  \end{enumerate}
  Let $H$ be any graph with $\Delta(H) \leq 2$ and let
  $Y_1, \ldots, Y_q$ be the vertex sets of added copies of $K_4$. We
  will modify $H$ and $Y_1, \ldots, Y_q$ to guarantee each of the
  properties (1)--(3) in turn, taking care not to invalidate earlier
  properties when establishing later ones, and taking care not to
  increase the size of $\cee$. At each step, we will observe that a
  $4$-coloring of the modified graph implies $4$-colorability of the
  original $G$.
  \begin{enumerate}[(1)]
  \item Let $H^1$ consist of the graph $H$
    where all edges within each $Y_j$ have been deleted, and let $G^1$ consist
    of $H^1$ with copies of $K_4$ glued in on the vertex sets $Y_1, \ldots, Y_j$.
    The only edges deleted in passing from $H$ to $H^1$ are restored when we glue
    in the $K_4$'s, so $G^1 = G$. Clearly, deleting edges cannot increase the size of $\cee$.
  \item If $H^1$ is 2-regular then let $H^2=H^1$. \gp{Otherwise, we
      form $H_2$ from $H_1$ by adding new vertices and edges as
      follows. Let $P_1, \ldots, P_t$ be the path components of
      $H^1$. For each path component $P_i$, if $\sizeof{V(P_i)}$ is
      odd then we add one new vertex $v_i$ adjacent to the endpoints
      of $P_i$, while if $\sizeof{V(P_i)}$ is even then we add two new
      adjacent vertices $v_i, w_i$ with $v_i$ adjacent to one endpoint
      of $P_i$ and $w_i$ adjacent to the other.} Let $G^2$ consist of
    $H^2$ with copies of $K_4$ glued in on the sets
    $Y_1, \ldots, Y_q$. We see that $G^2$ contains $G^1$ as an induced
    subgraph (since no new edges are added within $V(H^1)$), hence
    $4$-colorability of $G^2$ implies $4$-colorability of
    $G^1$. Furthermore, since all $Y_j \subseteq V(H_1)$, we see that
    Property~(1) still holds.  Since all new cycles created in this
    manner are even cycles, $\sizeof{\cee}$ has not increased in this
    step.
  \item Let $J$ be the subgraph of $H^2$
    induced by the vertices not covered by $Y_1, \ldots, Y_q$. Let $J'$
    be the disjoint union of $J$ and $t$ copies of  $K_3$, where $t$ is chosen so that $J'$
    has at least $12$ vertices and $\sizeof{V(J')}$ is divisible by $4$.
    (Since $3$ and $4$ are coprime, such a $t$ can always be found.)

    Let $k = \sizeof{V(J')}/4$. By our choice of $t$, we see that $k$
    is an integer with $k \geq 3 > \Delta(J')$. Hence, by Lemma~\ref{lem:equitable},
    we see that $J'$ has a $k$-coloring with color classes $A_1, \ldots, A_k$
    such that $\sizeof{A_i - A_j} \leq 1$. Since $\sizeof{V(J')}=4k$,
    this implies that all $\sizeof{A_i} = \sizeof{V(J')}/k = 4$. These color classes
    are independent sets that we will be able to glue new copies of $K_4$ onto.

    Let $H^3$ be the disjoint union of $H^2$ and $t$ copies of $K_3$
    (with the latter having the same vertex set as those we added in
    passing from $J$ to $J'$). Let $Y'_1, \ldots, Y'_{q'}$ consist of
    the original sets $Y_1, \ldots, Y_q$ together with the color
    classes $A_1, \ldots, A_k$ from the coloring of $J'$.  Let
    $G^3 = H^3$ with copies of $K_4$ glued in on the sets
    $Y'_1, \ldots, Y'_{q'}$.  As $G^3$ contains $G^2$ as an induced
    subgraph, $4$-colorability of $G^3$ implies $4$-colorability of
    $G^2$ and hence $G$.

    Now $Y'_1, \ldots, Y'_{q'}$ partition $V(H^3)$, and by our
    construction, each $Y'_j$ is an independent set in $H^3$. In passing
    from $H^2$ to $H^3$ we have maintained
    $2$-regularity, so Properties (1) and (2) still hold. As we have
    only added triangles in passing from $H^2$ to $H^3$, we also see
    that we have not increased $\sizeof{\cee}$.
  \end{enumerate}

  For the remainder of this proof, we will assume that Properties
  (1)--(3) hold.  We will view the edges of $G$ (and its subgraphs) as
  being colored red and blue: all edges coming from the 2-regular graph
  $H$ will be colored red and all edges of the added copies
  of $K_4$ will be colored blue. (Obviously this is not a proper
  edge-coloring, as in $G$ every vertex is incident to two red edges
  and three blue edges. However, Property~(1) does ensure that there are no red and blue edges in parallel).

  Let $\cee_0$ be the set of all odd cycles in $H$ (including triangles),
  and let $V(\cee_0)$ be the set of vertices in these odd cycles. Observe
  that every component of $H- V(\cee_0)$ is an even cycle, hence $H$ has
  a matching (of red edges) that saturates $V(H) - V(\cee_0)$. Fix such
  a matching $M_0$.

  \gp{A \emph{transversal} of $\cee_0$ is a vertex set
    $T \subset V(G)$ containing exactly one vertex from each cycle in
    $\cee_0$. We will use the word \emph{transversal} as shorthand for
    ``transversal of $\cee_0$''.}  \gp{Given a transversal $T$ and a
    vertex $t\in T$,} we write $C(t)$ for the cycle of $\cee_0$
  containing $t$. Observe that for any transversal $T$, there is a
  unique matching \gp{of $H$} that extends $M_0$ and saturates $V(G) - T$. Let
  $M(T)$ denote this unique matching. (While the base matching $M_0$
  is arbitrary, we use the same choice of $M_0$ for all transversals
  $T$ when defining $M(T)$.)

  Let $J$ be an arbitrary perfect matching of the blue edges in $G$,
  and for any transversal $T$, let $B(T) = M(T) \cup J$, considered as
  a subgraph of $G$. (As before, while the choice of $J$ is arbitrary, we use
  the same $J$ for all $T$.) Inheriting the edge coloring from $G$, we observe
  that\gp{, within $B(T)$,} every vertex $v$ is incident to exactly one blue edge
  and either exactly one red edge (if $v \notin T$) or to zero red
  edges (if $v \in T$). In particular, $B(T)$ is bipartite, and its
  components consist of  even cycles together with
  $\sizeof{\cee_0}/2$ paths whose endpoints are vertices in $T$.

  Now, among all the possible transversals $T$, we will choose an ``optimal'' transversal
  $T^*$. Our selection proceeds in two stages. First, among all transversals $T$, choose
  $T_1$ to minimize the sum of the lengths of all path components in $B(T_1)$.
  Call any transversal achieving this minimum a \emph{semi-optimal} transversal.

  \begin{claim}\label{claim:notQ}
    Let $T_1$ be a semi-optimal transversal, let $t_1 \in T_1$, and let $Q$
    be the component of $B(T_1)$ containing $t_1$. Then there is at least one
    vertex in the odd $H$-cycle $C(t_1)$ that is not contained in $Q$.
  \end{claim}
  \begin{proofclaim}
    Let $C_1 = C(t_1)$, and suppose to the contrary that every vertex of $C_1$ lies in $Q$.
    Observe that $Q$ is a path with $t_1$ as one endpoint and another
    vertex of $T_1$ as the other endpoint. Let $t_2$ be the other endpoint of $Q$,
    so that $Q$ is a $(t_1,t_2)$-path.

    Let $t'$ be the last vertex of $C_1$ along the $(t_1,t_2)$-path $Q$.
    Observe that since $Q$ contains every vertex of $C_1$, we have
    $t' \neq t_1$. Since $t'$ is the last vertex of $C_1$ along this path,
    the $(t',t_2)$-subpath does not contain any edges from $C_1$.

    Let $T'$ be the transversal obtained from $T_1$ by replacing $t_1$
    with $t'$. We claim that the path-components of $B(T')$ have a
    smaller sum of lengths than the path-components of $B(T_1)$,
    contradicting the choice of $T_1$.

    First observe that the only edges that lie in $B(T')$ but not in
    $B(T_1)$, or vice versa, are edges from the cycle $C_1$, since
    this is the only cycle whose transversal-representative has
    changed. Since (by hypothesis) every vertex of $C_1$ lies in $Q$, for every vertex $v$ not in $Q$ the set of edges incident to $v$ in $B(T_1)$ is identical to the set of edges incident to $v$ in $B(T')$. In particular, for every
    $t \in T - \{t_1, t_2\}$, the component containing $t$ in $B(T_1)$ is identical to the component containing $t$ in $B(T')$.

   Since the $(t',t_2)$-subpath
    of $Q$ did not use any edges of $C_1$, we see that it is still present
    in $B(T')$. However, the length of the $(t',t_2)$-path component in $B(T')$ is strictly
    less than the length of the $(t_1,t_2)$-path component in $B(T_1)$, with all other
    path components having the same length in both graphs. This contradicts the
    choice of $T_1$ and completes the proof of the claim.
  \end{proofclaim}
  Next, we refine our choice among the semi-optimal transversals.  To
  define our optimality criterion, first fix a cyclic orientation of
  each cycle in $\cee_0$.  For each $v \in V(\cee_0)$, let $v^+$ be
  the successor of $v$ in this cyclic orientation, and for $v,w$ in
  the same cycle, let $d^+(v,w)$ be the ``directed $H$-distance'' from $v$
  to $w$ along the directed $H$-edges.

  \begin{claim}\label{claim:nobb}
    Let $T_1$ be a semi-optimal transversal, let $t_1 \in T_1$, and
    let $Q$ be the component of $B(T_1)$ containing $t_1$. If $t_1^+ \in Q$,
    then the unique $(t_1,t_1^+)$-path in $Q$ starts with a blue edge
    and ends with a red edge.
  \end{claim}
  \begin{proofclaim}
    Since $t_1$ has no incident red edges in $B(T_1)$, the only
    other possibility is that the $(t_1,t_1^+)$-path starts and ends
    with a blue edge, as shown in Figure~\ref{fig:t1plus}(a).  As
    $t_1^+$ is covered by $M(T_1)$ and therefore has an incident red
    edge in $B(T_1)$, we see that the vertex following $t_1^+$ in
    $Q$ is its other $H$-neighbor, namely $t_1^{++}$.  We
      know that the other endpoint of $Q$ is another vertex of $T_1$;
    let $t_2 \in T_1$ be the other endpoint of $Q$.
    \begin{figure}
      \centering
      \begin{tikzpicture}
        \begin{scope}[yshift=0cm]
          \node at (1cm, .5cm) {(a)};
          \foreach \i in {1,...,5}
          {
            \apoint{} (a\i) at (2*\i cm, 0cm) {};
            \apoint{} (b\i) at (1cm + 2*\i cm, 0cm) {};
          }
          \bpoint{t_1} at (a1) {};
          \bpoint{t_1^{++}} at (a4) {};
          \bpoint{t_1^{+}} at (b3) {};
          \bpoint{t_2} at (b5) {};
          \begin{edges}
            \foreach \i in {1,...,5}
            {
              \draw[blueedge] (a\i) -- (b\i);
            }
            \foreach \i in {1,...,4}
            {
              \pgfmathsetmacro\j{1+\i};
              \draw[rededge] (b\i) -- (a\j);
            }
            \draw[reddots] (a1) .. controls ++(45:1.5cm) and ++(135:1.5cm) .. (b3);
          \end{edges}
        \end{scope}
        \begin{scope}[yshift=-2cm]
          \node at (1cm, .5cm) {(b)};
          \foreach \i in {1,...,5}
          {
            \apoint{} (a\i) at (2*\i cm, 0cm) {};
            \apoint{} (b\i) at (1cm + 2*\i cm, 0cm) {};
          }
          \bpoint{t_1} at (a1) {};
          \bpoint{t_1^+} at (a4) {};
          \bpoint{t_1^{++}} at (b3) {};
          \bpoint{t_2} at (b5) {};
          \begin{edges}
            \foreach \i in {1,...,5}
            {
              \draw[blueedge] (a\i) -- (b\i);
            }
            \foreach \i in {1,...,4}
            {
              \pgfmathsetmacro\j{1+\i};
              \draw[rededge] (b\i) -- (a\j);
            }
            \draw[reddots] (a1) .. controls ++(45:1.5cm) and ++(135:1.5cm) .. (a4);
          \end{edges}
        \end{scope}
      \end{tikzpicture}
      \caption{Two possible configurations of the component $Q$ of
        $B(T^*)$ containing $t_1$. Wavy lines denote blue
        edges in $B(T^*)$; solid lines denote red edges in $B(T^*)$;
        dashed edges are red edges contained in $H - M(T)$.}
      \label{fig:t1plus}
    \end{figure}
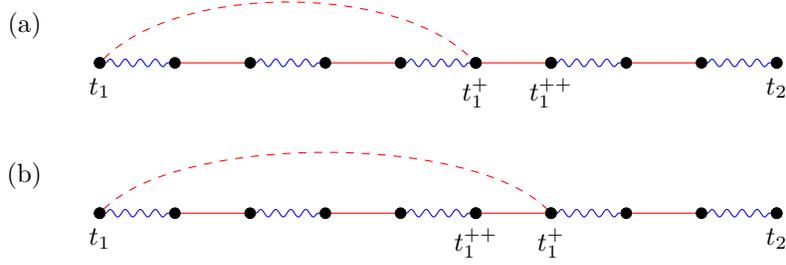

    Let $t' = t_1^{++}$ and let $T'$ be the transversal obtained from $T_1$ by replacing $t_1$ with $t'$.
    As shown in Figure~\ref{fig:rotate-matching} (see also Figure \ref{fig:t1plus}(a)), the only effect of this replacement
    on the matching $M(T_1)$ is to replace the matching-edge $t_1^+t_1^{++}$ with the matching-edge $t_1t_1^+$.
    This transforms the component $Q$ into an (even) cycle containing $t_1$ and $t_1^+$ and a shorter $(t_1^{++},t_2)$-path
    and has no effect on the other components of $B(T_1)$. Since the sum of the lengths of the path
    components in $B(T')$ is shorter than the sum for $B(T_1)$, this contradicts the choice of $T_1$
    as semi-optimal.
  \end{proofclaim}

  For a transversal $T$ and a vertex $t \in T$, let $Q$ denote the
  component of $B(T)$ containing $t$ (note that $Q$ is a path). Say
  $t \in T$ is \emph{bad} if $t^+ \in Q$. The \emph{cost} of a bad
  vertex $t \in T$ is the minimum directed $H$-distance $d^+(t,v)$ for
  $v \in V(C(t)) - V(Q)$, and the cost of a non-bad vertex $t \in T$ is $0$.
  (Note that for semioptimal $T$, the set $C(t) - Q$ is nonempty \gp{by Claim~\ref{claim:notQ}},
  hence the cost of a bad vertex is always finite.)

  Define the cost of a semi-optimal transversal $T$ to be the sum of the
  costs of the vertices in $T$. Among all semi-optimal transversals,
  choose $T^*$ to have minimum cost.

  \begin{claim}
    $T^*$ has no bad vertices.
  \end{claim}
  \begin{proofclaim}
    Suppose to the contrary that $t_1 \in T^*$ is a bad vertex, and
    let $Q$ be the component of $B(T^*)$ containing $t_1$.
    We know that $Q$ is a path whose other endpoint is another
    vertex of $T^*$; let $t_2 \in T^*$ be the other endpoint of $Q$.
    Let $v$ be the vertex of $C(t_1)-Q$ minimizing $d^+(t_1, v)$,
    so that $d^+(t_1, v)$ is the cost of the bad vertex $t_1$.

    Since $t_1$ is bad, we have $t_1^+ \in Q$. By Claim~\ref{claim:nobb}, the unique
    $(t_1, t_1^+)$-path in $Q$ starts with a blue edge and ends with
    a red edge, as shown in Figure~\ref{fig:t1plus}(b). In particular, the vertex preceding $t_1^+$
    in this path is the other $H$-neighbor of $t_1^+$, namely $t_1^{++}$. Thus, $d^+(t_1, v) \geq 3$.

    Let $t' = t_1^{++}$ and let $T'$ be the transversal obtained from $T^*$ by replacing $t_1$ with $t'$.
    As before, the only effect of this replacement
    on the matching $M(T^*)$ is to replace the matching-edge $t_1^+t_1^{++}$ with the matching-edge $t_1t_1^+$;
    see Figure~\ref{fig:rotate-matching}. In particular, this replacement does not alter the length of the
    path-component containing $t$ for any $t \in T^* - \{t_1,
    t_2\}$. Furthermore, after removing the edge $t^+_1t'$ and
    adding the edge $t^+_1t_1$, we see that $B(T')$ has a
    $(t',t_2)$-path using exactly the vertices of $Q$, obtained by
    starting at $t'$, traversing $Q$ backwards until $t_1$, taking the
    new edge $t_1t^+_1$, and then completing the rest of the path from
    $t^+_1$ to $t_2$ (see Figure \ref{fig:t1plus}(b)). Hence, the sum of the lengths of the
    path-components in $B(T')$ is the same as the sum of the lengths
    in $B(T^*)$, that is, $T'$ is also semi-optimal.
    \begin{figure}
      \centering
      \begin{tikzpicture}[scale=1.5]
        \begin{scope}[xshift=-2cm]
        \foreach \i in {1,...,7}
        {
          \pgfmathsetmacro\tht{360/7*\i + 90 - 360/7};
          \apoint{} (a\i) at (\tht : 1cm) {};
        }
        \apoint{t_1} at (a1) {};
        \rpoint{t_1^+} at (a7) {};
        \rpoint{t_1^{++}} at (a6) {};
        \draw[->, medge] (a3) -- (a2); \draw[->, medge] (a5) -- (a4); \draw[->, medge] (a7) -- (a6);
        \draw[->, rededge] (a1) -- (a7); \draw[->, rededge] (a2) -- (a1); \draw[->, rededge] (a4) -- (a3); \draw[->, rededge] (a6) -- (a5);
        \makecvert{a1};
      \end{scope}
      \begin{scope}[xshift=2cm]
        \foreach \i in {1,...,7}
        {
          \pgfmathsetmacro\tht{360/7*\i + 90 - 360/7};
          \apoint{} (a\i) at (\tht : 1cm) {};
        }
        \apoint{t_1} at (a1) {};
        \rpoint{t_1^+} at (a7) {};
        \rpoint{t_1^{++}} at (a6) {};
        \draw[->, medge] (a3) -- (a2); \draw[->, medge] (a5) -- (a4); \draw[->, medge] (a1) -- (a7);
        \draw[->, rededge] (a7) -- (a6); \draw[->, rededge] (a2) -- (a1); \draw[->, rededge] (a4) -- (a3); \draw[->, rededge] (a6) -- (a5);
        \makecvert{a6};
      \end{scope}
      \node[anchor=west] at (0cm, 0cm) {$\Longrightarrow$};
      \end{tikzpicture}
      \caption{Effect of replacing $t_1$ with $t' = t_1^{++}$ on the
        matching $M(T')$. Thick edges are edges in $M(T^*)$ or
        $M(T')$; circled vertices are vertices in the
          transversal.}
      \label{fig:rotate-matching}
    \end{figure}
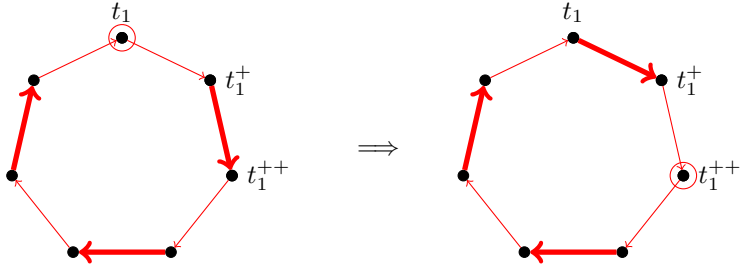

    Observe that since we replaced $t_1$ with $t' = t_1^{++}$ and
    since $d^+(t_1, v) \geq 3$, we have
    $d^+(t', v) \leq d^+(t_1, v) - 2$. Thus, the cost of $t'$ is
    strictly less than the cost of $t_1$ (regardless of whether $t'$
    is bad).  Furthermore, we have not altered any components of
    $B(T^*)$ except for the component $Q$, so every other vertex
    of $T'$ has the same cost it did in $T^*$. It follows that
    $T'$ is a semi-optimal transversal having lower cost than $T^*$,
    contradicting the choice of $T^*$.
  \end{proofclaim}
  Now we use the optimal transversal $T^*$ to produce the desired coloring.
  First we will randomly produce a $2$-coloring of $T^*$ using the colors
  black and white, then we will use the black-and-white coloring to $4$-color
  most of $G$.
  Each component of $B(T^*)$ is bipartite. Obtain a random
  black-and-white coloring of $B(T^*)$ by randomly choosing one of the
  two possible black-and-white colorings for each component
  independently and with equal probability. Let $\phi$ be the
  resulting coloring.
  \begin{claim}\label{claim:prob}
    For any vertex $t \in T^*$, the probability that $\phi(T)=\phi(t^+)$  is equal to $\tfrac{1}{2}$.
  \end{claim}
  \begin{proofclaim}
    By the previous claim, $t$ is not a bad vertex, so $t$ and $t^+$
    are in different components of $B(T^*)$. As the colorings on these
    components are chosen independently, \gp{the claim follows.}
  \end{proofclaim}
  Now we restrict our attention to the vertices of $T^*$ that lie in $\cee$, disregarding the vertices of $T^*$ that lie in triangles.
  Say that a vertex $t \in T^* \cap V(\cee)$ is \emph{unhappy} if $\phi(t) = \phi(t^+)$. By
  Claim \ref{claim:prob} and by linearity of expectation, the expected number of unhappy
  $T^*$-vertices in a random coloring is at most $\sizeof{T^* \cap V(\cee)}/2 = \sizeof{\cee}/2$. Hence,
  there is a coloring $\phi^*$ with at most $\sizeof{\cee}/2$ unhappy vertices
  in $T^* \cap V(\cee)$.

  Let $Z$ be the set of unhappy vertices for $\phi^*$. Let $W_1$ be the set of vertices colored black in
  $G-Z$, let $W_2$ be the set of vertices colored white in $G-Z$, and let
  $G_i$ be the induced subgraph $G[W_i]$ for $i\in\{1, 2\}$.

\begin{claim}\label{claim:redblue}
  Let $i \in \{1,2\}$. Every vertex of $G_i$ is incident, within $G_i$,
  to at most one red edge and to at most one blue edge.
  \end{claim}

\begin{proofclaim}
  First suppose that some $w \in V(G_i)$ were incident with two blue
  edges within $G_i$, say $wx$ and $wy$.  So, $w, x, y$ all received
  the same color under $\phi^*$. Since $B(T^*)$ is properly 2-colored
  by $\phi^*$, neither of $wx$ or $wy$ is in $B(T^*)$. Since $w, x, y$
  must all be part of the same (blue) $K_4$, and $B(T^*)$ contains a
  perfect (blue) matching of this $K_4$, we conclude that $xy$ is in
  $B(T^*)$. But this is a contradiction, since $x$ and $y$ receive the
  same color under $\phi^*$.

  \gp{Next, let some $w \in V(G_i)$ be given; we argue that $G$ is
    incident to at most one red edge within $G_i$. Let $wx$ and $wy$
    be the two red edges incident to $w$. If $w \notin T^*$, then one
    of these edges appears in $B(T^*)$, say $wx \in E(B(T^*))$.  Then
    $w$ and $x$ receive opposite colors in $B(T^*)$, so that
    $wx \notin E(G_i)$, establishing the claim that $w$ is incident to
    at most one red edge in $G_i$.}  However, if $w \in T^*$ we must
  be slightly more subtle.

  If $w \in T^* \cap V(\cee)$, then since $w \notin Z$, we see that
  $w$ is not unhappy under $\phi^*$. Hence,
  $\phi^*(w) \neq \phi^*(w^+)$, so the red edge $ww^+$ is not
  contained in $G_i$, \gp{again implying that $w$ is incident to at
    most one red edge in $G_i$.}

  Finally, if $w \in T^* - V(\cee)$, then $w$ lies in some
  triangle $wxy$ of $H$, so that $M(T^*)$ must contain the edge
  $xy$. Hence $x$ and $y$, the two neighbors of $w$ along red edges,
  receive opposite colors in the proper $2$-coloring of $B(T^*)$, \gp{implying
  that at most one of the red edges $wx, wy$ lies in $E(G_i)$.}
  \end{proofclaim}

 We are now able to deduce that $Z$ satisfies all the desired conclusions of Theorem \ref{thm:longodd-910}. By the definition of unhappy vertices, we know that $Z$ contains at most one
  vertex from each cycle in $\cee$. Our selection of $\phi^*$ (made possible by Claim \ref{claim:prob}) tells us that $\sizeof{Z} \leq \sizeof{\cee}/2$. Claim \ref{claim:redblue} tells us that each $G_i$ inherits a proper 2-edge-coloring in red and blue, and hence is bipartite. Since
  $V(G_1) \cup V(G_2) = V(G) - Z$, this implies that $\chi(G-Z) \leq 4$. Finally, since $|Z|\leq \tfrac{|\cee|}{2} \leq \tfrac{(|V(G)|/5)}{2}$, it follows that $|V(G)-Z|\geq \tfrac{9}{10} |V(G)|$.
\end{proof}

\section{Two lemmas about ISRs}\label{sec:isr-lems}

A \emph{total dominating set} in a graph $G$ is a set of vertices $X$
such that every vertex in $G$ is adjacent to a vertex in $X$. In
particular, every vertex of $X$ must also have a neighbor in $X$. The
\emph{total domination number} of $G$, written $\totdom(G)$, is the
size of a smallest total dominating set; if $G$ has isolated vertices,
then by convention we set $\totdom(G) = \infty$.

Given a graph $H$ and disjoint subsets
  $V_1, \ldots, V_n \subset V(H)$, for each $S \subset [n]$ we define a
  subgraph $H_S$ by taking the subgraph induced by the vertex set
  $\bigcup_{i \in S}V_i$ and deleting all edges in $G[V_i]$ for every $i\in S$.

  Using the above definitions, we can now state the following result
  of Haxell \cite{haxell-dom}.

\begin{theorem}[Haxell~\cite{haxell-dom}]\label{thm:haxell}
  Let $H$ be a graph and let $V_1, \ldots, V_n$ be disjoint subsets of
  $V(H)$. If, for all $S \subset [n]$, we have $\totdom(H_S) \geq 2\sizeof{S} - 1$, then $(V_1, \ldots, V_n)$ has an ISR.
\end{theorem}

Theorem \ref{thm:haxell} was originally stated in terms of hypergraphs
(see also \cite{king-haxell} for a formulation not in terms of
hypergraphs), and we have stated it above in a slightly modified but
equivalent formulation.

The first of the two lemmas we will prove in this section is a
deficiency version of Theorem \ref{thm:haxell}: using weaker bounds on
the size of total dominating sets, we can still obtain a ``large''
partial ISR. In particular, our proof will show that
Theorem~\ref{thm:haxell} is ``self-strengthening'', i.e., that the
following can be obtained as a Corollary to Theorem \ref{thm:haxell}
itself.

\begin{lemma}\label{lem:haxell-defic}
  Let $H$ be a graph, let $V_1, \ldots, V_n$ be disjoint subsets of
  $V(H)$, and let $k$ be a nonnegative integer. If, for all $S \subseteq [n]$, we have $\totdom(H_S) \geq 2\sizeof{S} - 1 - 2k$, then $(V_1, \ldots, V_n)$ has a
  partial ISR of size at least $n-k$.
\end{lemma}
\begin{proof}
  Let $H'$ be the disjoint union of $H$ with $k$ copies of $K_n$,
  and let $V'_1, \ldots, V'_n$ be obtained from $V_1, \ldots, V_n$
  by defining $V'_i$ to be $V_i$ together with one vertex from
  each copy of $K_n$, chosen so that $V'_1, \ldots, V'_n$ are
  disjoint.

  Observe that if $(V'_1, \ldots, V'_n)$ has an ISR, then at most $k$
  of the vertices from the ISR are from the added copies of $K_n$; the
  remaining $n-k$ vertices yield a partial ISR of $(V_1, \ldots, V_n)$
  of size at least $n-k$, as desired. Thus, it suffices to show that
  $(V'_1, \ldots, V'_n)$ has an ISR.

  To do this, we apply Theorem~\ref{thm:haxell}. Let $S$ be any subset
  of $[n]$. We will show that $\totdom(H'_S) \geq 2\sizeof{S}-1$. If
  $\sizeof{S} \leq 1$ then there is nothing to show, so assume that
  $\sizeof{S} \geq 2$.

  Let $H'_1$ be the subgraph of $H'_S$ induced by the original
  vertices of $H$ and let $H'_2$ be the subgraph of $H'_S$ induced by
  vertices from the added copies of $K_n$. Since each $V_i$
  has exactly one vertex from each added $K_n$, we see that $H'_2$
  is isomorphic to $k$ copies of $K_{\sizeof{S}}$.

  As there are no edges joining $H'_1$ and $H'_2$, clearly
  $\totdom(H') = \totdom(H'_1) + \totdom(H'_2)$. By hypothesis,
  $\totdom(H'_1) \geq 2\sizeof{S}-1-2k$, and since $\totdom(K_t) = 2$
  for all $t \geq 2$, we have $\totdom(H'_2) = 2k$. It follows that
  $\totdom(H') \geq 2\sizeof{S}-1$, and by Theorem~\ref{thm:haxell},
  it follows that $(V'_1, \ldots, V'_n)$ has an ISR.
\end{proof}

The second lemma we will prove in this section lets us ``combine'' ISRs
for two different families of disjoint sets, under suitable
conditions. In order to state it, we require the following technical
definition.

Let $\xee$ and $\yee$ be two collections of pairwise disjoint subsets
  of $V(G)$. However, we make no
  disjointness requirements between sets in $\xee$ and sets in $\yee$.
  The pair $(\xee, \yee)$ is \emph{admissible} for $G$ if, for each
  edge $e \in E(G)$, at least one of the following holds:
  \begin{itemize}
  \item There is some $X \in \xee$ with both endpoints of $e$ in $X$,
  \item There is some $Y \in \yee$ with both endpoints of $e$ in $Y$,
  \item Both endpoints of $e$ are missing from all $X \in \xee$, or
  \item Both endpoints of $e$ are missing from all $Y \in \yee$.
  \end{itemize}

\begin{lemma}\label{lem:combine-isr}
  Let $G$ be a graph, and suppose $(\xee, \yee)$ is an admissible
  pair for $G$. If
  $\xee$ has an ISR $R_{\xee}$ in $G$ and
  $\yee$ has an ISR $R_{\yee}$ in $G$, then $G$ has an
  independent set $R \subseteq R_{\xee} \cup R_{\yee}$ that is a transversal of
  both $\xee$ and $\yee$.
\end{lemma}
\begin{proof}
  Initially, let $R_0 = R_{\xee} \cup R_{\yee}$.  The set $R_0$
  clearly hits every $X_i$ and $Y_j$, but as there may be edges
  between $R_{\xee}$ and $R_{\yee}$, the set $R_0$ may not be
  independent. We next describe an algorithm for iteratively deleting
  vertices from $R_0$ in order to obtain an independent subset of
  $R_0$ that still hits every $X_i$ and every $Y_j$. Note that some
  vertices of $R_0$ may lie in $R_{\xee} \cap R_{\yee}$; such vertices
  are automatically isolated vertices in $R_0$.

To describe the algorithm, it will help to classify the edges between
$R_{\xee}$ and $R_{\yee}$. If $uv$ is an edge of $G$ with
$u \in R_{\xee}$ and $v \in R_{\yee}$, we say that $uv$ is an
\emph{$\xee$-edge} if $\{u,v\} \subseteq X_i$ for some $i$, and that
$uv$ is a \emph{$\yee$-edge} if $\{u,v\} \subseteq Y_j$ for some $j$.
Admissibility of the pair $(\xee, \yee)$ implies that any edge joining
a vertex of $R_{\xee}$ with a vertex of $R_{\yee}$ must be an
$\xee$-edge or a $\yee$-edge, since such edges intersect both a set in
$\xee$ and a set in $\yee$. It may be possible for an edge to be both
an $\xee$-edge and a $\yee$-edge.

\begin{claim}\label{claim:edgedeg}
  Every vertex of $R_{\xee}$ is incident to at most one $\yee$-edge, and every vertex
  of $R_{\yee}$ is incident to at most one $\xee$-edge.
\end{claim}
\begin{proofclaim}
  Suppose that $u \in R_{\xee}$ and $uv_1$, $uv_2$ are two different
  $\yee$-edges incident to $u$. It follows that
  $\{u, v_1\} \subseteq Y_1$ and $\{u, v_2\} \subseteq Y_2$ for
  some sets $Y_1$ and $Y_2$ in $\yee$, and since the sets $Y \in \yee$ are
  pairwise vertex-disjoint, this implies that $Y_1 = Y_2$. Hence, $v_1$ and
  $v_2$ lie in the same set $Y$. Since $R_{\yee}$ is an ISR of
  $\yee$ and $\{v_1, v_2\} \subseteq R_{\yee}$, this is a
  contradiction.

The same argument, interchanging the roles of $\xee$ and $\yee$, proves the claim about
$\xee$-edges.
\end{proofclaim}
Now consider the following algorithm, \gp{which defines a sequence of vertex sets $R_0, R_1, \ldots$} starting with
$R_0 = R_{\xee} \cup R_{\yee}$. Say a vertex $v$ is \emph{dangerous} for the vertex set $R_i$ if it
has degree $1$ in $G[R_i]$, and either $v \in R_{\xee}$ and the incident edge
is not an $\xee$-edge, or $v \in R_{\yee}$ and the incident edge is not a
$\yee$-edge. Thus, if $v \in R_{\xee}$ is dangerous \gp{for $R_i$}, then its incident edge in $G[R_i]$
is a $\yee$-edge, and vice versa; the awkward negative wording is
intended to exclude the possibility that the incident edge may be
\emph{both} an $\xee$-edge and a $\yee$-edge.

Note that vertices which were not initially dangerous \gp{for $R_0$}
may become dangerous \gp{for some later $R_j$} as their neighbors are
deleted, while vertices which were initially dangerous become
non-dangerous if their neighbor is deleted.  Given the set $R_{i}$, we
either produce a new set $R_{i+1}$ and proceed to the next round, or
produce the final set $R$:
\begin{itemize}
\item If $R_{i}$ has a dangerous vertex:
  \begin{itemize}
  \item Let $v$ be a vertex that is dangerous for $R_{i}$, and let $w$ be its unique neighbor in $R_{i}$.
  \item Let $R_{i+1} = R_{i} - w$, and proceed to the next round.
  \end{itemize}
\item Otherwise, if $R_{i}$ has no dangerous vertices, then let $R$ be obtained from $R_{i}$
  by deleting every vertex of $R_{\yee} \cap R_{i}$ that has
  positive degree in $G[R_{i}]$.
\end{itemize}
This algorithm clearly terminates, since
$\sizeof{R_{i+1}} < \sizeof{R_{i}}$ whenever $R_{i}$ has a dangerous
vertex, and the resulting set $R$ is clearly independent.  It remains
to show that $R$ hits every set $X \in \xee$ and every set
$Y \in \yee$.

First consider any set $X \in \xee$, and let $w$ be the representative
of $X$ in the set $R_{\xee}$. If $w$ is still in $R$, then clearly
$R \cap X \neq \emptyset$. Otherwise, $w \in R_{i+1} - R_{i}$ for some
$i$, which only occurs when $w$ is the sole neighbor of vertex $v$
which was dangerous in $R_{i}$.  (It cannot happen that $w$ is
deleted in the last step as a member of $R_{\yee} \cap R_{i}$, since
if $w\in R_{\yee}$ as well, then it is isolated in $R_0$ and thus in
every subsequent $R_j$, and hence never deleted in our
algorithm). Thus, $v$ is isolated in $R_{i+1}$ and every subsequent $R_j$,
and the algorithm therefore never deletes $v$ in the rest of its
execution. Hence $v \in R$. \gp{Furthermore, since $v$ was dangerous
  in $R_{i}$ with $v \in R_{\yee}$, we see that the edge $vw$ was an
  $\xee$-edge.}  This means that $\{v,w\} \subseteq X'$ for some
$X' \in \xee$, and since the sets $X \in \xee$ are pairwise disjoint,
this forces $X = X'$, so $v \in X$.  Hence $R \cap X \neq \emptyset$.

Next consider any set $Y \in \yee$, and let $w$ be the representative
of $Y$ in the set $R_{\yee}$. As before, if $w \in R$ then we are
done. Otherwise, $w$ was deleted at some point. If
$w \in R_{i} - R_{i+1}$ for some $i$ (that is, if $w$ was deleted
because of some dangerous vertex $v \in R_{\xee}$), then by the same
argument as before, we have $v \in R$ at the end of the algorithm, and
since $vw$ was a $\yee$-edge, we have $v \in Y$, so that
$R \cap Y \neq \emptyset$.

Otherwise, $w$ was deleted in the last step, when no dangerous
vertices remained in $R_{i}$. This implies that either $w$ had degree
at least $2$ in $R_i$, or $w$ had degree $1$ and the incident edge was
a $\yee$-edge (since if the incident edge were not a $\yee$-edge then
$w$ itself would be dangerous). Since, by Claim~\ref{claim:edgedeg},
the vertex $w$ is incident to at most one $\xee$-edge, in both cases
$w$ was incident to a $\yee$-edge in $R_i$. Let $vw$ be a $\yee$-edge
incident to $w$, with $v \in R_i$. Since $w$ was deleted in the last
step, no subsequent step could have deleted $v$, so $v \in R$ at the
end. Furthermore, since $vw$ is a $\yee$-edge, we have $v \in
Y$. Hence $R \cap Y \neq \emptyset$.

Thus, after executing the algorithm, $R$ is an independent set in $G$ that
intersects every set $X \in \xee$ and every set $Y \in \yee$.
\end{proof}

\section{Graphs with few triangles}\label{sec:few-tris}

Our goal in this section is to prove the following theorem.

\begin{theorem}\label{thm:triangle-11}
  Let $H$ be a graph with $\Delta(H)\leq 2$, and let $G$ be a
  graph obtained from $H$ by gluing in vertex-disjoint copies of
  $K_4$. Let $\mathcal{\tee}$ be the set of triangles in $H$.  Then there
  is a set of vertices $Z$ with $\sizeof{Z} \leq \sizeof{\tee}/4$,
  containing at most one vertex from each cycle in $H$, such that
  $\chi(G-Z) \leq 4$. It also holds that
  $\sizeof{V(G)-Z} \geq 11\sizeof{V(G)}/12$.
\end{theorem}

If $H$ has no triangles, then it has girth at least 4. As mentioned in
the introduction, Pei \cite{Pei} \gp{proved} that after gluing in
$K_4$'s to such an $H$, we get a graph $G$ that is 4-colorable. The
main idea of Pei's proof is to find an independent set $R$ hitting
each added $K_4$ and each cycle, and to observe that $G-R$ can then be
viewed as a subgraph of a ``cycle-plus-triangles'' graph. Since every
such graph is $3$-colorable (by the celebrated result of Fleischner
and Stiebitz \cite{cycletri}), using a fourth color on the set $R$
gives the desired $4$-coloring of $G$. We will adapt this idea in
order to $4$-color ``most'' of the vertices of $G$ in the case where
$H$ has few triangles.

Letting $\mathcal{T}$
denote the set of triangles in $H$ (where $\Delta(H)\leq 2$), we observe that
Pei's result easily yields a partial result itself. Deleting one edge
from each triangle yields a graph $H'$ with girth at least 4, and after $4$-coloring the resulting
``glued graph'' $G'$, we must, at worst, uncolor one vertex from each
triangle (an endpoint of a deleted edge) in order to obtain a proper
partial $4$-coloring (of at least $|V(G)-\mathcal{T}|\geq \tfrac{2}{3}|V(G)|$ vertices). We improve this to a partial coloring of $\tfrac{11}{12}$ of the vertices by showing that only $\sizeof{\tee}/4$
vertices must be deleted, rather than $\sizeof{\tee}$ vertices as in
the simple argument.

\begin{proof}[Proof of Theorem \ref{thm:triangle-11}]
  Let $X_1, \ldots, X_p$ be the vertex sets of the cycles of $H$ and
  let $Y_1, \ldots, Y_q$ be the vertex sets of the added copies of
  $K_4$. Let $\xee = (X_1, \ldots, X_p)$ and let
  $\yee = (Y_1, \ldots, Y_q)$.

  Note that $\yee$ has an ISR in $G$ by Theorem
  \ref{thm:haxell-maxdeg}. Our general strategy will be to first apply
  Lemma \ref{lem:haxell-defic} to find a ``large'' subfamily
  $\xee' \subseteq \xee$ that admits an ISR. We will then be able to
  apply Lemma~\ref{lem:combine-isr} to get an independent set $R$ in $G$
  which is a transversal of both $\xee'$ and $\yee$. Finally, we will
  use this $R$ (in a similar way to the set $R$ described above in
  Pei's proof) to define our desired 4-coloring.

\begin{claim}\label{claim:p-k} $\xee$ has a partial ISR with
  $p-k$ vertices, where $k=\lfloor \tfrac{\mathcal{T}}{4}\rfloor$.
\end{claim}
\begin{proofclaim}
  Let $S$ be any subset of $[p]$. In order for Lemma
  \ref{lem:haxell-defic} to yield our desired result, we must show
  that $\totdom(H_S) \geq 2\sizeof{S}-2k-1$. Since all edges
  induced by the sets $X_i$ are removed in constructing $H_S$, the only
  edges remaining in $H_S$ are edges that were added by the
  copies of $K_4$. Hence, every component of $H_S$ has at most $4$ vertices. Since every graph
    -- and, in particular, every component of $H_S$ --  has total domination number at least $2$, it follows that
  \[\totdom(H_S) \geq \sizeof{V(H_S)}/2.\] On the other hand, since
  every set $X_i$ is the vertex set of either a triangle or a cycle of
  length at least $4$, we have \[\sizeof{V(H_S)} \geq 4\sizeof{S} - \sizeof{\tee}.\]
  Combining these inequalities yields
  $$\totdom(H_S) \geq 2\sizeof{S} - \tfrac{\sizeof{\tee}}{2}=  2\sizeof{S} - 2\left(\tfrac{\sizeof{\tee}}{4}\right)\geq  2\sizeof{S} - 2\left(\floor{\tfrac{\sizeof{\tee}}{4}} + \tfrac{3}{4}\right) = 2\sizeof{S} - 2k - \tfrac{3}{2}.$$
  Since $\totdom(H_S)$ is an integer, this gives our desired bound.
  \end{proofclaim}

  Let $\xee'$ be the subfamily of $\xee$ consisting
  of the sets containing a vertex from the partial ISR found in Claim \ref{claim:p-k}. Then $\xee'$ has an ISR in $G$.

  \begin{claim}\label{claim:admissible} There is a independent set $R$ in $G$ that is a transversal of both $\xee'$ and of $\yee$.
  \end{claim}

  \begin{proofclaim} Since $\xee'$ and $\yee$ \gp{each} have an ISR in $G$, we get our desired result via
  Lemma~\ref{lem:combine-isr} provided that $(\xee', \yee)$
  is admissible for $G$. To this end, observe that
  every edge $e \in E(G)$ belongs to either a cycle of $H$ or one of the added $K_4$'s, and hence falls into one of the following categories:
  \begin{itemize}
  \item \gp{$e \in G[X]$ for some $X \in \xee - \xee'$, so that
    both endpoints of $e$ are missing from all $X \in \xee$, or}
  \item \gp{$e \in G[X]$ for some $X \in \xee'$, so that both endpoints of $e$ are in $X$, or}
  \item $e$ is an added edge from some $K_4$, hence both endpoints of $e$ are in  $Y$ for some $Y \in \yee$.
  \end{itemize}
  It follows that $(\xee', \yee)$ is admissible for $G$.
  \end{proofclaim}

  Let $F$ be a set consisting of one edge from each cycle of $H$ not
  represented in $\xee'$ (so $|F|=k$). Let $J = H-R-F$ and observe
  that $J$ is a graph of maximum degree at most $2$ with no
  cycles. By adding edges between the endpoints of path
  components in $J$, we obtain a Hamiltonian cycle $J'$ on the same
  vertex set.  For each $j \in [q]$, let $Y'_j = Y_j - R$; we have
  $\sizeof{Y'_j} = 3$ for all $j$, since $R$ intersects each $Y_j$ in
  exactly one vertex. Let $J^*$ be the graph obtained from $J'$ by
  gluing in a triangle on each $Y_j'$. By Fleischner and
  Stiebitz's \cite{cycletri} cycle + triangles result, we get that
  $J^*$ is $3$-colorable. As $G-F-R$ is a subgraph of $J^*$, it
  follows that $G-F-R$ is $3$-colorable. Using a fourth color on
  the independent set $R$ yields a $4$-coloring of $G-F$. Let $Z$
  be a vertex set consisting of one endpoint of each monochromatic
  edge in $F$. Now $G-Z$ is properly $4$-colored, and we have
  $|Z|\leq k\leq \tfrac{|\mathcal{T}|}{4}$. Since
    $|\mathcal{T}|\leq \tfrac{|V(G)|}{3}$, this implies
    $\sizeof{V(G)-Z} \geq 11\sizeof{V(G)}/12$. Furthermore, since $F$
    has at most one edge from each cycle, the vertex set $Z$ has at most
    one vertex from each cycle, as desired. Theorem \ref{thm:triangle-11} now follows.
\end{proof}

  As stated, Theorem~\ref{thm:triangle-11} gives us no control over
  which cycles contain uncolored vertices, in contrast to
  Theorem~\ref{thm:longodd-910} which guarantees that the uncolored
  vertices are contained in the long odd cycles of $H$. However, it is possible to refine the statement of Lemma~\ref{lem:haxell-defic} so that when we apply it in the proof of Theorem \ref{thm:triangle-11}, only ``dummy vertices'' are added to the sets $X_i$ obtained from triangles. With this change, one
  can guarantee that the set $R$ hits all long odd cycles, and that
  all uncolored vertices lie in triangles of $H$. Proving this
  formally would require more technical conditions in the hypothesis
  of Lemma~\ref{lem:haxell-defic}, so in the interest of clarity
  we have opted to only formally prove the simpler formulation.

\section{Theorems 1.2 and 1.3}\label{sec:final}

Combining Theorems \ref{thm:longodd-910} and \ref{thm:triangle-11} gives us an immediate proof of   Theorem~\ref{thm:main-cycles}.

\renewcommand\thesection{1}
\setcounter{proposition}{1}

\begin{theorem}
  Let $H$ be a graph with $\Delta(H) \leq 2$, and let $G$ be obtained
  from $H$ by gluing in vertex-disjoint copies of $K_4$. If $H$
  contains at most one odd cycle of length exceeding $3$, or if $H$
  contains at most $3$ triangles, then $\chi(G) \leq 4$.
\end{theorem}

\begin{proof}
If $H$ contains at most one odd cycle of length exceeding $3$, then we can apply Theorem \ref{thm:longodd-910} with $|\mathcal{C}|\leq 1$ to obtain
    a set of vertices $Z$ with $\sizeof{Z} \leq 1/2$ such that $G-Z$ is $4$-colorable.
    If $H$ contains at most $3$ triangles, then we can apply Theorem~\ref{thm:triangle-11} with $\sizeof{\tee} \leq 3$ to obtain a set of vertices $Z$ with $\sizeof{Z} \leq 3/4$ such that $G-Z$ is $4$-colorable. In either case, as $\sizeof{Z}$ is an integer, we have $Z = \emptyset$ and so $G$ is $4$-colorable.
\end{proof}

Theorem~\ref{thm:longodd-910} is weakest when most vertices of $H$ lie
in long odd cycles; Theorem~\ref{thm:triangle-11} is weakest when most
vertices of $H$ lie in triangles. These worst-case scenarios cannot
happen simultaneously; combining these bounds gives a stronger overall
bound on the number of vertices in a $4$-colorable subgraph, namely, Theorem \ref{thm:main-fraction}

\begin{theorem}
  Let $H$ be a graph with $\Delta(H) \leq 2$, and let $G$ be obtained
  from $H$ by gluing in vertex-disjoint copies of $K_4$. Then there
  is a set of vertices $Z$ with $\sizeof{Z} \leq \sizeof{V(G)}/22$
  such that $\chi(G-Z) \leq 4$.
\end{theorem}

\begin{proof}
  Let $\nt$ and $\nl$ denote the number of vertices of $H$
  that lie in triangles and in odd cycles of length exceeding $3$, respectively,
  and let $\nx$ denote the number of other vertices in $H$,  so that $|V(H)|=\sizeof{V(G)}= \nx +\nt + \nl$. Let $\nf$ denote the number of vertices of $G$ in a largest $4$-colorable induced subgraph.

  Theorem \ref{thm:longodd-910} says that $\nf \geq |V(G)|-\tfrac{|\mathcal{C}|}{2}$, where $\mathcal{C}$ is the set of all odd cycles in $H$ of length exceeding 3. Since $|\mathcal{C}|\leq \tfrac{\nl}{5}$, we get that
  \[\nf \geq \nx + \nt + \frac{9}{10}\nl.\]
  Similarly, Theorem \ref{thm:triangle-11} says that $\nf \geq |V(G)|-\tfrac{|\mathcal{T}|}{4}$, where $\mathcal{T}$ is the set of all triangles in $H$. Since $|\mathcal{T}|=\tfrac{\nt}{3}$, we get that
    \[\nf \geq \nx + \frac{11}{12}\nt + \nl.\]
  Given any $\lambda \in [0,1]$, we can take a convex combination of the above two inequalities (with $\lambda$ times the first and $(1-\lambda)$ times the second), to get
  \[ \nf \geq \nx + \left(\lambda + \frac{11}{12}(1-\lambda)\right)\nt +
    \left(\frac{9}{10}\lambda + (1-\lambda)\right)\nl. \]
  Setting $\lambda = 5/11$ equalizes the coefficients of $\nt$ and $\nl$, yielding the bound
  \[ \nf \geq \nx + \frac{21}{22}\nt + \frac{21}{22}\nl \geq \frac{21}{22}\sizeof{V(G)}.\qedhere \]
\end{proof}

\vspace*{.3in}

\bibliographystyle{amsplain} \bibliography{biblio}

\appendix

\renewcommand\thesection{}
\section{}

\renewcommand\thesection{1}
\setcounter{proposition}{6}

Here we present an elementary algorithmic proof of Theorem
\ref{thm:main-2isr}. Our main idea is similar to the proof of
Theorem~\ref{thm:longodd-910}: we'll have both red-coloured edges and
blue-coloured edges and we'd ideally want to partition the vertex set
into two parts so that the two induced graphs each have maximum degree
1 in red and maximum degree 1 in blue. (Such a partition would mean
that the two induced graphs are bipartite, and then the whole graph
would have chromatic number at most 4). We won't be able to show that
both induced graph are bipartite, but we'll \gp{modify} a lemma of
Haxell, Szab\'o, and Tardos (Lemma~2.6 of~\cite{HST}), to get a
stronger condition for one of the induced graphs: that it is not only
bipartite but its two partite sets are two ISRs.

 \begin{theorem}
  If $H$ is a graph with $\Delta(H)\leq 2$ and $V_1, \ldots, V_n$
  are disjoint subsets of $V(H)$ with each $\sizeof{V_i} = 4$,
  then $(V_1, \ldots, V_n)$ has two disjoint ISRs.
\end{theorem}

\begin{proof}
  By embedding $H$ in a larger graph we may assume that $V_1, V_2, \ldots, V_n$ partition $V(H)$.

  Let $G_1 = H$, let $G_2$ consist of the edges of a $C_4$ on each part $V_i$, and let $G_3$ consist of the edges of a $K_4$ on each part $V_i$. We regard the edges in $G_1$ as colored
  red and the edges in $G_3$ as colored blue (as in the proof of Theorem \ref{thm:longodd-910}); we shall also regard the edges of $G_2$ as colored green.  Let $G$ be the multigraph $G = G_1 \cup G_2 \cup G_3$.  (While the edge
  multiplicity is not relevant to the coloring problem, it simplifies
  things to be able to view a pair of vertices as possibly joined by
  edges of multiple colors.)

\begin{claim}\label{claim:HST} There exists a partition $V(G)$ into two sets $W_1$
    and $W_2$, with $|W_1|=|W_2|$, such that $G_i[W_i]$ has maximum
    degree $1$ for each $i\in\{1,2\}$ (that is, $G[W_1]$ has maximum
    degree $1$ with respect to red edges and $G[W_2]$ has maximum
    degree $1$ with respect to green edges).
\end{claim}

\begin{proofclaim}
  A lemma of Haxell, Szab\'o, and Tardos (Lemma~2.6 of~\cite{HST})
  tells us that we can find a partition of $V(G)$ into two sets $W_1$
  and $W_2$ such that $G_i[W_i]$ to has maximum degree $1$ for each
  $i\in\{1,2\}$ ($G[W_1]$ has maximum degree $1$ with respect to red
  edges and $G[W_2]$ has maximum degree $1$ with respect to green
  edges). We want to get their conclusion with, additionally, $\sizeof{W_1} = \sizeof{W_2}$.

  We form our partition using the following algorithm, which is
  adapted from \cite{HST} with one minor tweak (which we shall point
  out shortly). Fix an orientation of $G_1$ and of $G_2$ with
  maximum outdegree $1$ in each orientation. For each vertex
  $v \in V(G)$, let $v^{1+}$ and $v^{2+}$ denote the successor of $v$
  in $G_1$ and $G_2$, respectively, whenever these successors
  exist. Likewise, we write $v^{1-}$ and $v^{2-}$ for the predecessors
  of $v$ in $G_1$ and $G_2$ when these predecessors exist.

  Initially, set $W_1 = W_2 = \emptyset$.
  We start by adding an arbitrarily-chosen vertex to $W_2$,
  and after this and each subsequent added vertex, we choose
  the next vertex to add as follows:
  \begin{itemize}
  \item For $i=1,2$, if the just-added vertex $v$ was added to $W_i$:
    \begin{enumerate}[(i)]
    \item If $v^{i+}$ exists and is not yet placed, then add $v^{i+}$ to $W_{3-i}$. (Now $v^{i+}$ is the just-added vertex for the next step.)
    \item Otherwise, if $v^{i-}$ exists and is not yet placed, then add $v^{i-}$ to $W_{3-i}$. (Now $v^{i-}$ is the just-added vertex.)
    \item Otherwise, if there is any unplaced vertex $w$, then add $w$ to $W_{3-i}$. (Now $w$ is the just-added vertex.)
    \item Otherwise, terminate.
    \end{enumerate}
  \end{itemize}
  Our algorithm differs from the algorithm of \cite{HST} only in that
  our algorithm \emph{always} alternates between placing a vertex into
  $W_1$ or into $W_2$, while the algorithm of \cite{HST} always places
  the next vertex into $W_1$ when it makes an arbitrary choice in
  Case~(iii). Since $V_1, V_2, \ldots, V_n$ is a partition of $V(H)=V(G)$ where each part has size 4,  we know that
  $\sizeof{V(G)}$ is even, so this alternation guarantees
  $\sizeof{W_1} = \sizeof{W_2} = \sizeof{V(G)}/2$ at the end.

  The proof of Lemma~2.6 of~\cite{HST} immediately implies that
  $\Delta(G_i[W_i]) \leq 1$ for each $i$. (As their proof never
  specifically uses the choice of $W_1$ in Case~(iii), but rather only
  uses the choices made in Case~(i) and Case~(ii), it goes through
  without modification for this version of the algorithm.)
    \end{proofclaim}

    As a result of Claim \ref{claim:HST}, we can show that $G[W_1]$ satisfies some strong conditions.

   \begin{claim}\label{claim:bipW1} $G[W_1]$ has maximum degree $1$ in each of red and blue and it contains exactly two vertices from each $V_i$.
   \end{claim}

   \begin{proofclaim} Taking $W_1, W_2$ as in the previous claim,
   we must additionally show that:
  \begin{enumerate}[(1)]
  \item $G[W_1]$ has maximum degree 1 with respect to blue edges, and
  \item $W_1$ contains exactly two vertices from each $V_i$.
  \end{enumerate}

  If $W_1$ has least three vertices from some $V_i$,
  then this would force $G[W_1]$ to have maximum degree at least 2 with respect to blue edges. Hence, it suffices only to prove (1). To this end, suppose on the
  contrary that some vertex $v \in W_1$ is incident to two blue edges
  within $G[W_1]$. Then $W_1$
  contains at least $3$ vertices of the corresponding copy of $K_4$,
  and $W_2$ contains at most $1$ vertex of that copy of $K_4$. Since
  $\sizeof{W_1} = \sizeof{W_2} = \sizeof{V(G)}/2$, this forces $W_2$
  to contain at least $3$ vertices of some other copy of $K_4$. However this would force $G[W_2]$ to have a vertex with green degree at least two, contradicting our choice of $W_2$.
\end{proofclaim}

Claim \ref{claim:bipW1} implies that not only is $G[W_1]$ bipartite, but that its partite sets (each of which must contain exactly one vertex from each $V_i$ since $G$ has a blue $K_4$ induced on each $V_i$)   are our desired pair of ISRs.
\end{proof}

\end{document}